\documentclass[11pt]{article}

\usepackage[margin=1in]{geometry}
\usepackage[T1]{fontenc}
\usepackage{lmodern}
\usepackage{amsmath,amssymb,amsthm,mathtools}
\usepackage{enumitem}
\usepackage[colorlinks=true,linkcolor=blue,citecolor=blue,urlcolor=blue]{hyperref}
\usepackage{microtype}

\hypersetup{
  pdftitle={Extremality and Limit Laws for the Siblings of the Coupon Collector},
  pdfauthor={Christopher D. Long},
  pdfsubject={Coupon collector problem; siblings model; stochastic ordering; limit laws},
  pdfkeywords={coupon collector, siblings, stochastic ordering, exponential order statistics, Poissonization, occupancy problems}
}

\title{Extremality and Limit Laws for the Siblings of the Coupon Collector}
\author{Christopher D. Long\thanks{Headlamp Software. Email: \texttt{galizur@gmail.com}.}}
\date{}

\newtheorem{theorem}{Theorem}
\newtheorem{lemma}{Lemma}
\newtheorem{proposition}{Proposition}
\newtheorem{corollary}{Corollary}
\newtheorem{conjecture}{Conjecture}
\newtheorem{hypothesis}{Hypothesis}
\newtheorem{remark}{Remark}

\newcommand{\E}{\mathbb{E}}
\newcommand{\Pbb}{\mathbb{P}}
\newcommand{\R}{\mathbb{R}}

\newcommand{\dd}{\,d}
\newcommand{\st}{\mathrm{st}}

\newcommand{\Exp}{\mathrm{Exp}}
\newcommand{\GammaDist}{\mathrm{Gamma}}
\newcommand{\Cov}{\operatorname{Cov}}
\newcommand{\Var}{\operatorname{Var}}
\newcommand{\Corr}{\operatorname{Corr}}
\newcommand{\e}{\mathrm e}
\newcommand{\ind}[1]{\mathbf 1_{\{#1\}}}

\allowdisplaybreaks

\begin{document}
\maketitle

\begin{abstract}
We study the siblings version of the coupon collector problem. A main collector stops when every coupon type has appeared at least once, duplicates are passed successively to later siblings, and $U_j^N$ denotes the number of empty spaces in collector $j$'s album at the main completion time. We prove three results. First, for every fixed $N$ and $j\ge2$, $\E U_j^N$ is uniquely maximized over positive coupon distributions by the uniform distribution; in fact it decreases strictly along every nonconstant ray from the uniform vector. Second, in the uniform model, $U_j^N$ is stochastically increasing in $N$, and we construct an increasing coupling using top spacings of exponential order statistics. Third, for fixed album indices $2,\ldots,J$, the naturally normalized vector converges jointly to $(W,\ldots,W)$, where $W$ is exponential with mean one. We also derive exact Poissonized and alternating-subset formulae and give a transfer principle for leading expectation asymptotics.
\end{abstract}

\noindent\textbf{Keywords.} Coupon collector problem; stochastic ordering; exponential order statistics; Poissonization; occupancy problems; limit theorem.

\medskip
\noindent\textbf{2020 Mathematics Subject Classification.} Primary 60C05; Secondary 60F05, 60G55, 60E15, 05A16.

\section{Introduction}

The siblings of the coupon collector are defined as follows. There are $N$ coupon types with probabilities
\[
        p=(p_1,\ldots,p_N),\qquad p_i>0,\qquad \sum_{i=1}^N p_i=1.
\]
A main collector samples coupons independently with replacement until every type has appeared at least once. A new coupon goes into the main collector's album. A duplicate is passed to the next collector in the sibling line; if it is a duplicate for that collector, it is passed on again, and so forth.

We use the following indexing convention throughout. Collector $1$ is the main collector. Collector $2$ is the first sibling, collector $3$ is the second sibling, and in general collector $j$ is the $(j-1)$st sibling for $j\ge2$. Thus $j$ is always a collector or album index, not a sibling number. For $j\ge2$, let $U_j^N$ denote the number of empty spaces in collector $j$'s album at the time collector $1$ completes her album; equivalently, $U_j^N$ counts coupon types that have appeared fewer than $j$ times by that completion time.

The classical coupon collector problem and its multiple-cover variants go back at least to standard treatments such as Feller \cite{Feller} and to the double Dixie cup problem of Newman and Shepp \cite{NewmanShepp}; see also Holst \cite{Holst} and the survey of Boneh and Hofri \cite{BonehHofri}. The siblings version was studied by Pintacuda \cite{Pintacuda}, Foata, Han, and Lass \cite{FoataHanLass}, Foata and Zeilberger \cite{FoataZeilberger}, Adler, Oren, and Ross \cite{AdlerOrenRoss}, Ross \cite{Ross}, and Doumas and Papanicolaou \cite{DoumasPapanicolaouSiblings}. In the equal-probability case, known formulae include
\begin{equation}\label{eq:equal-finite-intro}
        \E U_j^N=\sum_{k=1}^N \binom Nk (-1)^{k+1}k^{-(j-1)},\qquad j\ge2,
\end{equation}
with leading asymptotic
\begin{equation}\label{eq:equal-leading-intro}
        \E U_j^N\sim \frac{(\log N)^{j-1}}{(j-1)!}.
\end{equation}
For the first sibling $j=2$ in the equal-probability case, Papanicolaou and Doumas \cite{PapanicolaouDoumas} proved the distributional limit
\[
        \frac{U_2^N}{\log N}\Rightarrow \Exp(1).
\]
One consequence of the spacing representation below is a short probabilistic proof of this theorem, an extension to every fixed album index, and a joint limit showing that the normalized counts for fixed album indices share the same asymptotic random amplitude. The citation to this preprint is for attribution and comparison; the proof of the corresponding case is included below as part of the joint theorem.

For unequal probabilities, Adler, Oren, and Ross obtained the exact integral formula
\begin{equation}\label{eq:AOR-intro}
        \E_p U_j^N=\sum_{k=1}^N\int_0^\infty p_k\e^{-p_kt}\frac{(p_kt)^{j-1}}{(j-1)!}\prod_{i\ne k}(1-\e^{-p_it})\dd t.
\end{equation}
Doumas and Papanicolaou used this formula to obtain detailed asymptotics for broad families of unequal probabilities and formulated the following finite-$N$ extremal conjecture.

\begin{conjecture}[Doumas--Papanicolaou maximum conjecture]\label{conj:DP-max}
For fixed $N$ and $j\ge2$, the quantity $\E_p U_j^N$ is maximized over the positive probability simplex at the uniform vector
\[
        u=(1/N,\ldots,1/N).
\]
\end{conjecture}

In the same paper they observed that, in the equal-probability case, $\E U_j^N$ is increasing in $N$, and conjectured the stronger stochastic monotonicity
\begin{equation}\label{eq:stoch-conj-intro}
        U_j^N\le_{\st} U_j^{N+1},\qquad j\ge2.
\end{equation}
These two conjectural statements are of different types. Conjecture~\ref{conj:DP-max} is an extremal inequality over probability vectors for fixed $N$. The stochastic monotonicity statement compares equal-probability models with different alphabet sizes.

The main purpose of this paper is to prove both conjectural statements and to add the corresponding joint fixed-index limit law. Thus the uniform distribution is extremal at fixed $N$, the equal-probability residual count is monotone in the alphabet size, and the fixed-index residual vector has a one-dimensional exponential limit. The proofs use standard Poissonization, exponential order statistics, and finite-dimensional positivity. The exact formula \eqref{eq:AOR-intro} is first derived directly from independent Poisson processes. The same representation yields an alternating subset formula, which is the finite-dimensional starting point for the radial extremality proof. In the equal-probability case, top spacings of exponential order statistics give both the monotone coupling in $N$ and the joint limit theorem.

Compared with the preceding work, the finite-$N$ results below play a different role from the asymptotic formulae. The alternating subset identity gives a direct proof of the exact extremal inequality over all positive probability vectors, while the exponential-spacing construction gives a coupling that resolves the stochastic monotonicity conjecture and simultaneously identifies the joint fixed-index limit. The endpoint transfer theorem is included as a complementary asymptotic tool and is not used in the proofs of the finite-$N$ extremality or monotonicity results.

\paragraph{Note added after initial submission.}
After the initial version of this manuscript was submitted for publication, Doumas and Spektor posted the preprint \emph{Equal probabilities maximize the expected deficit in the siblings of the coupon collector} \cite{DoumasSpektor}. Their main theorem independently proves the finite-$N$ radial expectation-extremality result corresponding to Theorem~\ref{thm:radial-max} and Corollary~\ref{cor:max-conj} below. Their proof uses a separable integral representation, one integration by parts, and Chebyshev's covariance inequality. The proof given here is independent and uses the alternating subset formula and positive pair kernels. The remaining results of the present paper, including the stochastic monotonicity theorem, the nested exponential-spacing coupling, the joint fixed-index limit theorem, mixed-moment consequences, and the transfer theorem, are not contained in \cite{DoumasSpektor}.

\subsection*{Summary of results}

The results are as follows.
\begin{enumerate}[label=(\arabic*)]
\item We give a self-contained Poissonized derivation of the exact integral identity \eqref{eq:AOR-intro}, and an exact finite subset formula
\[
        \E_p U_j^N=\sum_{\varnothing\ne B\subseteq[N]}(-1)^{|B|-1}\sum_{k\in B}\left(\frac{p_k}{p_B}\right)^j,
        \qquad p_B:=\sum_{i\in B}p_i.
\]
This formula recovers \eqref{eq:equal-finite-intro} immediately. In the equal-probability case it also gives exact harmonic polynomial identities for the expectations.

\item We prove a transfer theorem for leading expectations. If the first-missing-type mass converges on a scale $B_N+C_Nx$ and the normalized marked density converges on the same scale, then the normalized expectation of $U_j^N$ converges to
\[
        \int_{-\infty}^{\infty}h_j(x)\e^{-\Lambda(x)}\dd x.
\]
The Gumbel case $\Lambda(x)=\e^{-x}$ and $h_j(x)=\e^{-x}$ gives the leading asymptotic by a universal integral equal to one.

\item We recover the equal-probability leading asymptotic \eqref{eq:equal-leading-intro} and the endpoint-Laplace leading term
\[
        \E U_j^N\sim \frac1{(j-1)!}\left(\log \frac{f(N)}{f'(N)}\right)^{j-1}
\]
for the decaying weight regime $w_k=1/f(k)$ under explicit endpoint-Laplace and tail-envelope hypotheses. In the standard decaying-probability examples this identifies the leading term of Doumas and Papanicolaou's three-term expansion.

\item We prove Conjecture~\ref{conj:DP-max} in the stronger radial form: for every nonuniform $p$,
\[
        \theta\mapsto \E_{u+\theta(p-u)}U_j^N
\]
is strictly decreasing on $(0,1]$. The derivative factors as
\[
        -\sum_{a<b}(p_a-p_b)(p_a^{j-1}-p_b^{j-1})K_{ab}^{(j)}(p),
\]
where every pair kernel $K_{ab}^{(j)}(p)$ is strictly positive.

\item In the equal-probability case, we prove the joint fixed-index distributional limit
\[
        \left(\frac{U_2^N}{a_{N,2}},\frac{U_3^N}{a_{N,3}},\ldots,\frac{U_J^N}{a_{N,J}}\right)
        \Rightarrow (W,W,\ldots,W),\qquad
        a_{N,j}=\frac{(\log N)^{j-1}}{(j-1)!},\quad W\sim\Exp(1).
\]
For $J=2$ this recovers the theorem of Papanicolaou and Doumas \cite{PapanicolaouDoumas}; for $J>2$ it gives the corresponding extension to fixed album indices $2,\ldots,J$. We also obtain convergence of fixed mixed moments, including variance, covariance, and asymptotic perfect correlation.

\item We prove the stochastic monotonicity conjecture \eqref{eq:stoch-conj-intro}. More precisely, for each fixed $j\ge2$ there is a coupling under which
\[
        U_j^2\le U_j^3\le U_j^4\le\cdots\qquad\text{almost surely}.
\]
The proof uses the independent top spacings of exponential order statistics. The almost-sure coupling also gives immediate increasing-transform and increasing-convex-order consequences; see Corollary~\ref{cor:order-consequences}.
\end{enumerate}

\section{Poissonization and exact formulae}

Let $N_i(t)$ be independent Poisson processes with rates $p_i$. The merged process has rate one, and the sequence of labels at its jump times is an iid sequence with distribution $p$. Thus the coupon counts at jump times have the same distribution as in the original discrete collector.

Let
\[
        E_i=\inf\{t:N_i(t)\ge1\},\qquad G_{i,j}=\inf\{t:N_i(t)\ge j\}.
\]
The main collector completes at
\[
        X_1=\max_{1\le i\le N}E_i.
\]
At this time coupon $i$ is missing from the $j$th collector's album if and only if $N_i(X_1)<j$, equivalently $G_{i,j}>X_1$.

\begin{proposition}[Exact Poissonized identity]\label{prop:poissonized}
For every $N\ge1$, every probability vector $p$, and every $j\ge2$,
\begin{equation}\label{eq:poissonized}
        \E_p U_j^N=\sum_{k=1}^N\int_0^\infty p_k\e^{-p_kt}\frac{(p_kt)^{j-1}}{(j-1)!}\prod_{i\ne k}(1-\e^{-p_it})\dd t.
\end{equation}
\end{proposition}

\begin{proof}
Since $j\ge2$, $G_{k,j}\ge E_k$. Hence
\[
        \{G_{k,j}>X_1\}=\left\{G_{k,j}>\max_{i\ne k}E_i\right\}.
\]
The random variable $G_{k,j}$ has density
\[
        p_k\e^{-p_kt}\frac{(p_kt)^{j-1}}{(j-1)!},
\]
and is independent of $\{E_i:i\ne k\}$. Therefore
\[
\begin{aligned}
        \Pbb(G_{k,j}>X_1)
        &=\int_0^\infty p_k\e^{-p_kt}\frac{(p_kt)^{j-1}}{(j-1)!}
          \Pbb\left(\max_{i\ne k}E_i<t\right)\dd t \\
        &=\int_0^\infty p_k\e^{-p_kt}\frac{(p_kt)^{j-1}}{(j-1)!}
          \prod_{i\ne k}(1-\e^{-p_it})\dd t.
\end{aligned}
\]
Summing over $k$ proves the identity.
\end{proof}

\begin{lemma}[Rate normalization]\label{lem:rate-normalization}
For positive rates $r=(r_1,\ldots,r_N)$ define
\[
        I_j(r)=\sum_{k=1}^N\int_0^\infty r_k\e^{-r_kt}\frac{(r_kt)^{j-1}}{(j-1)!}\prod_{i\ne k}(1-\e^{-r_it})\dd t.
\]
Then $I_j(cr)=I_j(r)$ for every $c>0$. Consequently, if $p_i=r_i/\sum_\ell r_\ell$, then $\E_pU_j^N=I_j(r)$.
\end{lemma}

\begin{proof}
The identity $I_j(cr)=I_j(r)$ follows by the change of variables $s=ct$. Taking $c=(\sum_i r_i)^{-1}$ and applying Proposition~\ref{prop:poissonized} gives the final assertion.
\end{proof}

\begin{proposition}[Alternating subset formula]\label{prop:subset}
Let $p_B=\sum_{i\in B}p_i$. Then, for every $j\ge2$,
\begin{equation}\label{eq:subset}
        \E_pU_j^N=\sum_{\varnothing\ne B\subseteq[N]}(-1)^{|B|-1}\sum_{k\in B}\left(\frac{p_k}{p_B}\right)^j.
\end{equation}
\end{proposition}

\begin{proof}
Expand the product in \eqref{eq:poissonized}:
\[
        \prod_{i\ne k}(1-\e^{-p_it})=\sum_{A\subseteq[N]\setminus\{k\}}(-1)^{|A|}\e^{-p_At}.
\]
Then
\[
\begin{aligned}
        \E_p U_j^N
        &=\sum_{k=1}^N\sum_{A\subseteq[N]\setminus\{k\}}(-1)^{|A|}
          \int_0^\infty p_k\e^{-(p_k+p_A)t}\frac{(p_kt)^{j-1}}{(j-1)!}\dd t \\
        &=\sum_{k=1}^N\sum_{A\subseteq[N]\setminus\{k\}}(-1)^{|A|}\frac{p_k^j}{(p_k+p_A)^j},
\end{aligned}
\]
because $\int_0^\infty \e^{-st}t^{j-1}\dd t=(j-1)!s^{-j}$. Now put $B=A\cup\{k\}$. Then $|A|=|B|-1$ and $p_k+p_A=p_B$, giving \eqref{eq:subset}.
\end{proof}

\begin{corollary}[Equal-probability finite formula]\label{cor:equal-finite}
If $p_i=1/N$ for all $i$, then
\[
        \E U_j^N=\sum_{r=1}^N(-1)^{r-1}\binom Nr r^{-(j-1)}.
\]
\end{corollary}

\begin{proof}
For a set $B$ of size $r$, $p_k/p_B=1/r$ for all $k\in B$. Hence
\[
        \sum_{k\in B}\left(\frac{p_k}{p_B}\right)^j=r\cdot r^{-j}=r^{-(j-1)}.
\]
Summing over the $\binom Nr$ sets of size $r$ gives the formula.
\end{proof}

\begin{corollary}[Harmonic-polynomial form]\label{cor:harmonic}
Let $H_N^{(m)}=\sum_{r=1}^N r^{-m}$ and let $h_m$ denote the complete homogeneous symmetric polynomial of degree $m$. In the equal-probability model,
\[
        \E U_j^N=h_{j-1}\left(1,\frac12,\ldots,\frac1N\right).
\]
In particular,
\[
        \E U_2^N=H_N,
\]
\[
        \E U_3^N=\frac12\left(H_N^2+H_N^{(2)}\right),
\]
and
\[
        \E U_4^N=\frac16\left(H_N^3+3H_NH_N^{(2)}+2H_N^{(3)}\right).
\]
\end{corollary}

\begin{proof}
The generating function of the complete homogeneous symmetric polynomials is
\[
        \sum_{m\ge0}h_m\left(1,\frac12,\ldots,\frac1N\right)z^m
        =\prod_{r=1}^N\frac1{1-z/r}
        =\frac{N!}{(1-z)(2-z)\cdots(N-z)}.
\]
The partial fraction decomposition is
\[
        \frac{N!}{(1-z)(2-z)\cdots(N-z)}
        =\sum_{r=1}^N(-1)^{r-1}\binom Nr\frac{r}{r-z}.
\]
Expanding each term as a power series in $z$ gives
\[
        h_m\left(1,\frac12,\ldots,\frac1N\right)=\sum_{r=1}^N(-1)^{r-1}\binom Nr r^{-m}.
\]
Taking $m=j-1$ and using Corollary~\ref{cor:equal-finite} proves the identity. The displayed formulae for $j=2,3,4$ are the first Newton identities for complete homogeneous symmetric polynomials.
\end{proof}

\section{A transfer theorem for leading expectations}

We state a transfer theorem for leading expectations. Its hypotheses are intentionally modular: they separate the first-completion probability from the marked density associated with the residual count.

Fix $j\ge2$. For each $N$, let $p_{N,1},\ldots,p_{N,N}$ be positive rates, not necessarily normalized to sum to one, and let $U_j^N$ denote the residual count for the probability vector obtained by normalizing these rates. By Lemma~\ref{lem:rate-normalization}, the Poissonized integral depends only on the ratios of the rates. Let $B_N\in\R$, $C_N>0$, and let $A_{N,j}>0$ be an abstract normalizing constant. Put
\[
        t_N(x)=B_N+C_Nx.
\]
Assume $B_N/C_N\to\infty$, so that the lower endpoint $t=0$ corresponds to $x=-B_N/C_N\to-\infty$.

Define the first-defect mass
\[
        M_N(x)=\sum_{i=1}^N\e^{-p_{N,i}t_N(x)},
\]
the atomlessness term
\[
        S_N(x)=\sum_{i=1}^N\e^{-2p_{N,i}t_N(x)},
\]
and the level-$j$ marked density
\[
        H_{N,j}(x)=C_N\sum_{i=1}^Np_{N,i}\e^{-p_{N,i}t_N(x)}\frac{(p_{N,i}t_N(x))^{j-1}}{(j-1)!}.
\]

\begin{theorem}[Transfer theorem for leading expectations]\label{thm:transfer}
Suppose that the following conditions hold.
\begin{enumerate}[label=(\roman*)]
\item For every compact $K\subset\R$, $M_N\to\Lambda$ uniformly on $K$.
\item For every compact $K\subset\R$, $S_N\to0$ uniformly on $K$.
\item For every compact $K\subset\R$,
\[
        \frac{H_{N,j}(x)}{A_{N,j}}\to h_j(x)
\]
uniformly on $K$.
\item The limiting integral is finite:
\[
        \int_{-\infty}^{\infty}h_j(x)\e^{-\Lambda(x)}\dd x<\infty.
\]
\item The normalized tails are negligible:
\[
\begin{aligned}
        \lim_{L\to\infty}\limsup_{N\to\infty}\frac1{A_{N,j}}
        \int_{\{x:|x|>L,\ t_N(x)>0\}}
        &C_N\sum_{i=1}^Np_{N,i}\e^{-p_{N,i}t_N(x)}\frac{(p_{N,i}t_N(x))^{j-1}}{(j-1)!}\\
        &\times \prod_{\ell\ne i}(1-\e^{-p_{N,\ell}t_N(x)})\dd x=0.
\end{aligned}
\]
\end{enumerate}
Then, where $\Lambda$ is the compact limit in (i),
\[
        \frac{\E U_j^N}{A_{N,j}}\to\int_{-\infty}^{\infty}h_j(x)\e^{-\Lambda(x)}\dd x.
\]
In particular, if $\Lambda(x)=\e^{-x}$ and $h_j(x)=\e^{-x}$, then $\E U_j^N\sim A_{N,j}$.
\end{theorem}

\begin{proof}
Write
\[
        q_{N,i}(x)=\e^{-p_{N,i}t_N(x)}.
\]
After changing variables $t=t_N(x)$ in \eqref{eq:poissonized},
\[
        \E U_j^N=\int_{-B_N/C_N}^{\infty}\sum_{i=1}^N C_Np_{N,i}q_{N,i}(x)\frac{(p_{N,i}t_N(x))^{j-1}}{(j-1)!}P_{N,i}(x)\dd x,
\]
where
\[
        P_{N,i}(x)=\prod_{\ell\ne i}(1-q_{N,\ell}(x)).
\]
Fix a compact $K\subset\R$. Since $S_N\to0$ uniformly on $K$, also $\max_i q_{N,i}\to0$ uniformly on $K$. Therefore, uniformly for $i$ and $x\in K$,
\[
\begin{aligned}
        \log P_{N,i}(x)
        &=\sum_{\ell\ne i}\log(1-q_{N,\ell}(x))\\
        &=-\sum_{\ell\ne i}q_{N,\ell}(x)+O\left(\sum_{\ell\ne i}q_{N,\ell}(x)^2\right)\\
        &=-M_N(x)+q_{N,i}(x)+O(S_N(x))=-M_N(x)+o(1).
\end{aligned}
\]
Hence $P_{N,i}(x)=\e^{-\Lambda(x)}+o(1)$ uniformly in $i$ and $x\in K$. Thus, uniformly on $K$,
\[
        \frac1{A_{N,j}}\sum_{i=1}^N C_Np_{N,i}q_{N,i}(x)\frac{(p_{N,i}t_N(x))^{j-1}}{(j-1)!}P_{N,i}(x)
        \to h_j(x)\e^{-\Lambda(x)}.
\]
Integrating over $K$ gives convergence on compact sets. The tail assumption and the integrability of $h_j\e^{-\Lambda}$ allow $K=[-L,L]$ and then $L\to\infty$. This proves the first claim. In the Gumbel case,
\[
        \int_{-\infty}^{\infty}\e^{-x}\e^{-\e^{-x}}\dd x=1
\]
by the substitution $y=\e^{-x}$. The final claim follows.
\end{proof}

\section{Leading expectation asymptotics}

\subsection*{Equal probabilities}

Let $p_{N,i}=1/N$. Choose
\[
        B_N=N\log N,
        \qquad C_N=N,
        \qquad t_N(x)=N(\log N+x).
\]
Then
\[
        M_N(x)=N\e^{-t_N(x)/N}=\e^{-x},
        \qquad S_N(x)=N\e^{-2t_N(x)/N}=N^{-1}\e^{-2x}\to0.
\]
Moreover
\[
\begin{aligned}
        H_{N,j}(x)
        &=N\sum_{i=1}^N\frac1N\e^{-(\log N+x)}\frac{(\log N+x)^{j-1}}{(j-1)!}\\
        &=\e^{-x}\frac{(\log N+x)^{j-1}}{(j-1)!}.
\end{aligned}
\]
Thus with
\[
        a_{N,j}=\frac{(\log N)^{j-1}}{(j-1)!},
\]
one has $H_{N,j}(x)/a_{N,j}\to\e^{-x}$ uniformly on compact sets. In this application of the transfer theorem, $A_{N,j}=a_{N,j}$. The remaining tail condition is elementary.

\begin{lemma}[Equal-probability tails]\label{lem:equal-tails}
For the equal-probability scaling above, the normalized tail condition in Theorem~\ref{thm:transfer} holds.
\end{lemma}

\begin{proof}
After changing variables, the normalized transfer integrand is, for $x>-\log N$,
\[
        I_N(x)=\e^{-x}\left(1+\frac{x}{\log N}\right)^{j-1}\left(1-\frac{\e^{-x}}{N}\right)^{N-1}.
\]
For $x\ge0$ and $N\ge3$,
\[
        I_N(x)\le \e^{-x}(1+x)^{j-1},
\]
whose right tail is integrable uniformly in $N$. For $-\log N<x\le0$,
\[
        \left(1+\frac{x}{\log N}\right)^{j-1}\le1
\]
and
\[
        \left(1-\frac{\e^{-x}}N\right)^{N-1}\le \exp\left\{-\frac{N-1}{N}\e^{-x}\right\}\le\exp\left\{-\frac12\e^{-x}\right\}
\]
for all $N\ge2$. Hence
\[
        I_N(x)\le \e^{-x}\exp\left\{-\frac12\e^{-x}\right\},\qquad -\log N<x\le0,
\]
and this envelope has a vanishing left tail. Combining the two estimates proves the normalized tail negligibility.
\end{proof}

\begin{corollary}[Equal-probability leading asymptotic]\label{cor:equal-leading}
For every fixed $j\ge2$,
\[
        \E U_j^N\sim \frac{(\log N)^{j-1}}{(j-1)!}.
\]
\end{corollary}

\begin{proof}
Apply Theorem~\ref{thm:transfer} with $A_{N,j}=a_{N,j}$ and $\Lambda(x)=h_j(x)=\e^{-x}$.
\end{proof}

\subsection*{Endpoint-Laplace arrays}

We next give a modular endpoint version which contains the leading term of the decaying probability regime studied by Doumas and Papanicolaou. This subsection is independent of the finite-$N$ extremality, stochastic-monotonicity, and joint-limit results proved below. By Lemma~\ref{lem:rate-normalization}, we may work in the scale-invariant rate normalization
\[
        w_{N,k}=\frac1{f(k)},\qquad 1\le k\le N,
\]
which corresponds to the probability vector $p_{N,k}=w_{N,k}/W_N$, where $W_N=\sum_{k\le N}w_{N,k}$.

Let $f\in C^1([1,\infty))$ be positive and increasing, with $f'(x)>0$ eventually and $f(N)\to\infty$. Put
\[
        \beta_N=\frac{f'(N)}{f(N)},
        \qquad \rho_N=\log\frac1{\beta_N}=\log\frac{f(N)}{f'(N)},
\]
and assume $0<\beta_N<1$ eventually and $\rho_N\to\infty$. Finally set
\[
        r_{N,k}=\frac{f(N)}{f(k)},\qquad u_N(x)=\rho_N-\log\rho_N+x.
\]

\begin{hypothesis}[Endpoint-Laplace and tail envelopes]\label{hyp:endpoint}
Fix $j\ge2$. The following endpoint estimates hold.
\begin{enumerate}[label=(\roman*)]
\item For every fixed real $a\ge0$ and uniformly for bounded $x$,
\begin{equation}\label{eq:endpoint-i}
        \sum_{k=1}^N r_{N,k}^a\exp\{-u_N(x)r_{N,k}\}\sim\frac{\e^{\rho_N-u_N(x)}}{u_N(x)}.
\end{equation}
\item The same estimate with $u_N(x)$ replaced by $2u_N(x)$ holds for $a=0$.
\item There is a constant $C>0$ such that, for all sufficiently large $N$ and all $x\ge0$,
\begin{equation}\label{eq:endpoint-right-tail}
        \sum_{k=1}^N r_{N,k}^j\e^{-u_N(x)r_{N,k}}
        \le C\left(\frac{\e^{\rho_N-u_N(x)}}{u_N(x)}+\e^{-u_N(x)}\right).
\end{equation}
\item The left endpoint has a negligible global envelope:
\begin{equation}\label{eq:endpoint-left-tail}
\begin{aligned}
        \lim_{L\to\infty}\limsup_{N\to\infty}\int_{\{x<-L: u_N(x)>0\}}
        &\left(\frac{u_N(x)}{\rho_N}\right)^{j-1}\sum_{k=1}^N r_{N,k}^j\e^{-u_N(x)r_{N,k}}\\
        &\times \exp\left\{-\sum_{\ell=1}^N\e^{-u_N(x)r_{N,\ell}}+\e^{-u_N(x)}\right\}\dd x=0.
\end{aligned}
\end{equation}
\end{enumerate}
\end{hypothesis}

\begin{remark}
The compact asymptotic conditions in Hypothesis~\ref{hyp:endpoint} are the usual one-sided endpoint Laplace estimates. The right-tail bound \eqref{eq:endpoint-right-tail} and the left-tail envelope \eqref{eq:endpoint-left-tail} are used only to justify tail negligibility in the transfer theorem. The additive $\e^{-u_N(x)}$ term accounts for the endpoint summand $k=N$ when $x$ is very large. In standard smooth endpoint classes they follow from the same monotone endpoint-Laplace comparison that gives \eqref{eq:endpoint-i}; the contributing window has width of order $(u_N\beta_N)^{-1}$ below $N$, while the far-left product term gives exponential damping.
\end{remark}

\begin{lemma}[Endpoint tail negligibility]\label{lem:endpoint-tail}
Assume Hypothesis~\ref{hyp:endpoint}. With
\[
        t_N(x)=f(N)u_N(x),\qquad C_N=f(N),\qquad A_{N,j}=\frac{\rho_N^{j-1}}{(j-1)!},
\]
the normalized tail condition in Theorem~\ref{thm:transfer} holds.
\end{lemma}

\begin{proof}
In the rate normalization $w_{N,k}=1/f(k)$, the normalized transfer integrand is
\[
        I_N(x)=\frac1{A_{N,j}}\sum_{k=1}^N r_{N,k}\e^{-u_N(x)r_{N,k}}\frac{(u_N(x)r_{N,k})^{j-1}}{(j-1)!}
        \prod_{\ell\ne k}(1-\e^{-u_N(x)r_{N,\ell}}),\qquad u_N(x)>0.
\]
Since the product is at most one, \eqref{eq:endpoint-right-tail} gives
\[
        I_N(x)\le\frac{H_{N,j}(x)}{A_{N,j}}
        \le C\e^{-x}\left(\frac{u_N(x)}{\rho_N}\right)^{j-2}
        +C\beta_N\rho_N\e^{-x}\left(\frac{u_N(x)}{\rho_N}\right)^{j-1}.
\]
For $x\ge0$, $u_N(x)/\rho_N\le1+x$ for all large $N$, and $\beta_N\rho_N\le1$ eventually. Hence
\[
        I_N(x)\le C_j(1+x)^{j-1}\e^{-x},\qquad x\ge0,
\]
with the interpretation $(1+x)^0=1$ when it occurs. This envelope has a vanishing right tail.

For the left tail, use
\[
        \prod_{\ell\ne k}(1-\e^{-u_N(x)r_{N,\ell}})
        \le \exp\left\{-\sum_{\ell\ne k}\e^{-u_N(x)r_{N,\ell}}\right\}.
\]
Because $r_{N,k}\ge1$, the omitted term is at most $\e^{-u_N(x)}$. Therefore the normalized left-tail integral is bounded above by the expression in \eqref{eq:endpoint-left-tail}, which tends to zero as $L\to\infty$. The two tail bounds prove the normalized tail condition.
\end{proof}

\begin{theorem}[Endpoint leading term]\label{thm:endpoint}
Assume Hypothesis~\ref{hyp:endpoint}. Then for every fixed $j\ge2$,
\[
        \E U_j^N\sim \frac{\rho_N^{j-1}}{(j-1)!}
        =\frac1{(j-1)!}\left(\log\frac{f(N)}{f'(N)}\right)^{j-1}.
\]
\end{theorem}

\begin{proof}
Apply the transfer theorem in the positive-rate normalization $p_{N,k}=w_{N,k}=1/f(k)$, which is equivalent to the normalized coupon probabilities by Lemma~\ref{lem:rate-normalization}. Use the scaling from Lemma~\ref{lem:endpoint-tail}. The first-defect mass is
\[
        M_N(x)=\sum_{k=1}^N\e^{-u_N(x)r_{N,k}}.
\]
By Hypothesis~\ref{hyp:endpoint} with $a=0$,
\[
        M_N(x)\sim\frac{\e^{\rho_N-u_N(x)}}{u_N(x)}.
\]
Since $u_N(x)=\rho_N-\log\rho_N+x$,
\[
        \e^{\rho_N-u_N(x)}=\rho_N\e^{-x},\qquad u_N(x)\sim\rho_N,
\]
and therefore $M_N(x)\to\e^{-x}$ uniformly on compact $x$-sets. The estimate with $2u_N(x)$ gives
\[
        S_N(x)=\sum_{k=1}^N\e^{-2u_N(x)r_{N,k}}\to0
\]
uniformly on compact sets.

The marked density equals
\[
\begin{aligned}
        H_{N,j}(x)
        &=f(N)\sum_{k=1}^N\frac1{f(k)}\e^{-u_N(x)r_{N,k}}\frac{(u_N(x)r_{N,k})^{j-1}}{(j-1)!}\\
        &=\frac{u_N(x)^{j-1}}{(j-1)!}\sum_{k=1}^N r_{N,k}^j\e^{-u_N(x)r_{N,k}}.
\end{aligned}
\]
By Hypothesis~\ref{hyp:endpoint} with $a=j$,
\[
        H_{N,j}(x)\sim \frac{u_N(x)^{j-1}}{(j-1)!}\frac{\e^{\rho_N-u_N(x)}}{u_N(x)}
        \sim\frac{\rho_N^{j-1}}{(j-1)!}\e^{-x}
\]
uniformly on compact sets. Lemma~\ref{lem:endpoint-tail} verifies the tail condition, and
\[
        \int_{-\infty}^{\infty}\e^{-x}\e^{-\e^{-x}}\dd x=1.
\]
The result follows from Theorem~\ref{thm:transfer}.
\end{proof}

The endpoint hypotheses are verified for the two standard decaying arrays in Proposition~\ref{prop:standard-endpoint} below. Thus the corresponding applications of Theorem~\ref{thm:endpoint} are unconditional.

\begin{corollary}[Generalized Zipf law]\label{cor:zipf}
Let $\alpha>0$ and
\[
        p_{N,k}=\frac{k^{-\alpha}}{\sum_{\ell=1}^N\ell^{-\alpha}},\qquad 1\le k\le N.
\]
Then, for every fixed $j\ge2$,
\[
        \E U_j^N\sim\frac{(\log N)^{j-1}}{(j-1)!}.
\]
\end{corollary}

\begin{proof}
Apply Theorem~\ref{thm:endpoint} and Proposition~\ref{prop:standard-endpoint} with $f(x)=x^\alpha$. Then
\[
        \rho_N=\log\frac{f(N)}{f'(N)}=\log\frac N\alpha=\log N+O(1),
\]
which gives the displayed asymptotic.
\end{proof}

\begin{corollary}[Stretched exponential decay]\label{cor:stretch}
Let $\lambda>0$, $0<q<1$, and
\[
        p_{N,k}=\frac{\e^{-\lambda k^q}}{\sum_{\ell=1}^N\e^{-\lambda\ell^q}},\qquad 1\le k\le N.
\]
Then, for every fixed $j\ge2$,
\[
        \E U_j^N\sim\frac{(1-q)^{j-1}(\log N)^{j-1}}{(j-1)!}.
\]
\end{corollary}

\begin{proof}
Apply Theorem~\ref{thm:endpoint} and Proposition~\ref{prop:standard-endpoint} with $f(x)=\e^{\lambda x^q}$. Since
\[
        \frac{f'(N)}{f(N)}=\lambda qN^{q-1},
        \qquad \rho_N=(1-q)\log N-\log(\lambda q),
\]
the result follows.
\end{proof}

\section{\texorpdfstring{Finite-$N$ extremality}{Finite-N extremality}}

We now prove the finite-$N$ maximum conjecture. The proof is exact and does not use asymptotics.

\begin{theorem}[Strict radial maximum at uniform]\label{thm:radial-max}
Fix $N\ge2$ and $j\ge2$. Let
\[
        \Phi_j(p)=\E_pU_j^N,
        \qquad u=(1/N,\ldots,1/N).
\]
For every nonuniform probability vector $p$, the map
\[
        \theta\mapsto \Phi_j(u+\theta(p-u))
\]
is strictly decreasing on $(0,1]$. Consequently,
\[
        \Phi_j(p)<\Phi_j(u)
\]
for every nonuniform $p$.
\end{theorem}

\begin{proof}
Fix a nonuniform $p$ and put
\[
        h_i=p_i-\frac1N,
        \qquad p_i(\theta)=\frac1N+\theta h_i,
        \qquad 0\le\theta\le1.
\]
For $0<\theta\le1$, write $p_i=p_i(\theta)$ to lighten notation. For a nonempty subset $B\subseteq[N]$, let
\[
        p_B=\sum_{i\in B}p_i,
        \qquad \alpha_i=\frac{p_i}{p_B}\quad(i\in B).
\]
By Proposition~\ref{prop:subset},
\[
        \Phi_j(p(\theta))=\sum_{\varnothing\ne B\subseteq[N]}(-1)^{|B|-1}\sum_{i\in B}\alpha_i^j.
\]
Let $r=|B|$ and $h_B=\sum_{i\in B}h_i$. Since
\[
        \alpha_i'=\frac{h_i-\alpha_ih_B}{p_B},
\]
and since
\[
        h_i=\frac{p_i-1/N}{\theta},
        \qquad h_B=\frac{p_B-r/N}{\theta},
\]
we have
\[
        h_i-\alpha_ih_B=\frac{r\alpha_i-1}{N\theta}.
\]
Therefore
\[
\begin{aligned}
        \frac{\dd}{\dd\theta}\sum_{i\in B}\alpha_i^j
        &=\frac{j}{N\theta p_B}\sum_{i\in B}\alpha_i^{j-1}(r\alpha_i-1)\\
        &=\frac{j}{N\theta p_B}\sum_{a<b,\ a,b\in B}(\alpha_a-\alpha_b)(\alpha_a^{j-1}-\alpha_b^{j-1}).
\end{aligned}
\]
The last identity is the standard pairwise identity
\[
        \sum_{a<b}(x_a-x_b)(y_a-y_b)
        =r\sum_i x_iy_i-\left(\sum_i x_i\right)\left(\sum_i y_i\right)
\]
with $x_i=\alpha_i$ and $y_i=\alpha_i^{j-1}$.

Now group the derivative by unordered pairs $\{a,b\}$. If $B=\{a,b\}\cup R$ with $R\subseteq[N]\setminus\{a,b\}$, then
\[
        \frac1{p_B}(\alpha_a-\alpha_b)(\alpha_a^{j-1}-\alpha_b^{j-1})
        =\frac{(p_a-p_b)(p_a^{j-1}-p_b^{j-1})}{p_B^{j+1}}.
\]
Since $|B|-1=|R|+1$, we obtain
\begin{equation}\label{eq:radial-deriv}
        \frac{\dd}{\dd\theta}\Phi_j(p(\theta))
        =-\frac{j}{N\theta}\sum_{1\le a<b\le N}(p_a-p_b)(p_a^{j-1}-p_b^{j-1})K_{ab}^{(j)}(p),
\end{equation}
where
\[
        K_{ab}^{(j)}(p)=\sum_{R\subseteq[N]\setminus\{a,b\}}(-1)^{|R|}\frac1{(p_a+p_b+p_R)^{j+1}}.
\]
It remains to check that $K_{ab}^{(j)}(p)>0$. Using
\[
        \frac1{s^{j+1}}=\frac1{j!}\int_0^\infty t^j\e^{-st}\dd t,
\]
we get
\begin{equation}\label{eq:kernel-positive}
\begin{aligned}
        K_{ab}^{(j)}(p)
        &=\frac1{j!}\int_0^\infty t^j\e^{-(p_a+p_b)t}\sum_{R\subseteq[N]\setminus\{a,b\}}(-1)^{|R|}\e^{-p_Rt}\dd t\\
        &=\frac1{j!}\int_0^\infty t^j\e^{-(p_a+p_b)t}\prod_{\ell\ne a,b}(1-\e^{-p_\ell t})\dd t>0.
\end{aligned}
\end{equation}
Finally, the function $x\mapsto x^{j-1}$ is strictly increasing on $(0,\infty)$, so
\[
        (p_a-p_b)(p_a^{j-1}-p_b^{j-1})\ge0,
\]
with equality if and only if $p_a=p_b$. Since $p(\theta)$ is nonuniform for $\theta>0$, at least one pair contributes strictly. Equation \eqref{eq:radial-deriv} and positivity \eqref{eq:kernel-positive} imply
\[
        \frac{\dd}{\dd\theta}\Phi_j(p(\theta))<0
\]
for all $0<\theta\le1$. Hence $\Phi_j(p(\theta))$ is strictly decreasing on $(0,1]$. Since $\Phi_j$ is continuous at $\theta=0$, we obtain
\[
        \Phi_j(p)=\Phi_j(p(1))<\Phi_j(p(0))=\Phi_j(u).
\]
This proves the theorem.
\end{proof}

\begin{corollary}[Resolution of the maximum conjecture]\label{cor:max-conj}
For every $N\ge2$, every $j\ge2$, and every positive probability vector $p$,
\[
        \E_pU_j^N\le \E_uU_j^N,
\]
with equality if and only if $p=u$.
\end{corollary}

\begin{proof}
This is Theorem~\ref{thm:radial-max}, with equality only for the uniform vector.
\end{proof}

\begin{remark}[Comparison with Doumas--Spektor]\label{rem:DS-comparison}
The finite-$N$ radial expectation-extremality theorem was independently obtained by Doumas and Spektor \cite{DoumasSpektor} after the initial submission of this manuscript. Their proof starts from a separable integral representation and converts the radial derivative into a weighted covariance, whose sign follows from Chebyshev's correlation inequality. The proof above instead uses the finite alternating subset formula and yields the positive pair-kernel factorization \eqref{eq:radial-deriv}--\eqref{eq:kernel-positive}. These approaches are independent. The other results in the present paper concern monotone couplings, stochastic order in the alphabet size, joint fixed-index limits, mixed moments, and leading-asymptotic transfer; they are not consequences of the Doumas--Spektor theorem.
\end{remark}

\begin{corollary}[Strict local maximality]\label{cor:hessian}
The Hessian of $\Phi_j$ at $u$ is strictly negative definite on the tangent space of the simplex. More precisely, if $\sum_i h_i=0$, then
\[
        D^2\Phi_j(u)[h,h]
        =-\frac{j(j-1)N^3}{j!}\left[\int_0^\infty x^j\e^{-2x}(1-\e^{-x})^{N-2}\dd x\right]\sum_{i=1}^Nh_i^2.
\]
\end{corollary}

\begin{proof}
Differentiate \eqref{eq:radial-deriv} at $\theta=0$, or equivalently expand Proposition~\ref{prop:subset} to second order at $u$. For a subset $B$ of size $r$, put $H_B=\sum_{i\in B}h_i$. The normalized coordinates in $B$ satisfy
\[
        \frac{\dd}{\dd\varepsilon}\frac{u_i+\varepsilon h_i}{\sum_{\ell\in B}(u_\ell+\varepsilon h_\ell)}\bigg|_{\varepsilon=0}
        =\frac Nr\left(h_i-\frac{H_B}{r}\right).
\]
Thus the second variation of the $B$ term is
\[
        j(j-1)N^2r^{-j}\left(\sum_{i\in B}h_i^2-\frac{H_B^2}{r}\right).
\]
Summing over subsets of size $r$ gives
\[
        \sum_{|B|=r}\left(\sum_{i\in B}h_i^2-\frac{H_B^2}{r}\right)
        =\frac Nr\binom{N-2}{r-2}\sum_i h_i^2,
\]
using $\sum_i h_i=0$. Hence
\[
        D^2\Phi_j(u)[h,h]
        =j(j-1)N^3\left(\sum_i h_i^2\right)\sum_{r=2}^N(-1)^{r-1}\binom{N-2}{r-2}\frac1{r^{j+1}}.
\]
Finally,
\[
\begin{aligned}
        \sum_{r=2}^N(-1)^{r-1}\binom{N-2}{r-2}\frac1{r^{j+1}}
        &=\frac1{j!}\int_0^\infty x^j\sum_{r=2}^N(-1)^{r-1}\binom{N-2}{r-2}\e^{-rx}\dd x\\
        &=-\frac1{j!}\int_0^\infty x^j\e^{-2x}(1-\e^{-x})^{N-2}\dd x.
\end{aligned}
\]
This proves the formula and strict negativity for $h\ne0$.
\end{proof}

\section{Uniform probabilities: monotonicity and limits}

We now prove the second conjectural statement. The proof gives an explicit monotone coupling.

\begin{theorem}[Nested exponential-spacing coupling]\label{thm:coupling}
Fix $j\ge2$. Let $\xi_1,\xi_2,\ldots$ be independent random variables with
\[
        \xi_r\sim\Exp(r),
\]
and let $G_{1,j},G_{2,j},\ldots$ be iid $\GammaDist(j-1,1)$ random variables, independent of the $\xi_r$'s. Define
\[
        \widetilde U_j^N=1+\sum_{\ell=1}^{N-1}\ind{G_{\ell,j}>\xi_1+\cdots+\xi_\ell}.
\]
Then $\widetilde U_j^N$ has the same distribution as $U_j^N$ in the equal-probability $N$-coupon model. Consequently,
\[
        \widetilde U_j^2\le \widetilde U_j^3\le \widetilde U_j^4\le\cdots\qquad\text{almost surely}.
\]
\end{theorem}

\begin{proof}
Poissonize the equal-probability model and, after multiplying time by $N$, let the $N$ coupon processes have independent rate-one arrivals. Let
\[
        E_1,\ldots,E_N
\]
be their first arrival times. These are iid $\Exp(1)$ random variables. Write their order statistics as
\[
        E_{1:N}<\cdots<E_{N:N}.
\]
The main collector completes at $E_{N:N}$. The last coupon is always missing from the $j$th album, because it has appeared exactly once at the completion time.

Consider the coupon whose first arrival is $\ell$ places below the maximum, i.e. whose first arrival time is $E_{N-\ell:N}$, where $1\le\ell\le N-1$. Its age at completion is
\[
        A_\ell^{(N)}=E_{N:N}-E_{N-\ell:N}.
\]
By the strong Markov property, the additional waiting time after its first arrival until its $j$th arrival is $\GammaDist(j-1,1)$ and is independent of the first-arrival order statistics and of the corresponding variables for the other coupons. Thus this coupon contributes to $U_j^N$ if and only if
\[
        G_{\ell,j}>A_\ell^{(N)}.
\]
For iid exponential order statistics, the spacings
\[
        E_{N:N}-E_{N-1:N},\quad E_{N-1:N}-E_{N-2:N},\quad \ldots,\quad E_{2:N}-E_{1:N}
\]
are independent exponentials with rates $1,2,\ldots,N-1$, respectively. Therefore
\[
        A_\ell^{(N)}\stackrel{d}=\xi_1+\cdots+\xi_\ell.
\]
This representation is consistent as $N$ varies: the first $N-1$ top-spacing variables are the same variables $\xi_1,\ldots,\xi_{N-1}$. Hence
\[
        U_j^N\stackrel{d}=1+\sum_{\ell=1}^{N-1}\ind{G_{\ell,j}>\xi_1+\cdots+\xi_\ell}=\widetilde U_j^N.
\]
The almost-sure monotonicity is immediate from adding one nonnegative indicator when $N$ is increased by one.
\end{proof}

\begin{corollary}[Resolution of stochastic monotonicity]\label{cor:stoch-monotone}
For every fixed $j\ge2$,
\[
        U_j^N\le_{\st}U_j^{N+1}\qquad (N\ge2)
\]
in the equal-probability model.
\end{corollary}

\begin{corollary}[Order consequences of the nested coupling]\label{cor:order-consequences}
For every fixed $j\ge2$, the equal-probability sequence $U_j^N$ is increasing in the almost-sure, usual stochastic, increasing-convex, and increasing-transform orders. Equivalently, under the coupling of Theorem~\ref{thm:coupling}, $\widetilde U_j^N\le\widetilde U_j^{N+1}$ almost surely, and consequently, for every increasing function $f$ for which the expectations exist,
\[
        \E f(U_j^N)\le \E f(U_j^{N+1}).
\]
In particular, for $z>1$,
\[
        \E z^{U_j^N}\le \E z^{U_j^{N+1}},
\]
and for $s>0$,
\[
        \E\e^{-sU_j^N}\ge \E\e^{-sU_j^{N+1}}.
\]
\end{corollary}

\begin{proof}
The almost-sure inequality in Theorem~\ref{thm:coupling} implies the displayed expectation inequality for every increasing $f$ by applying $f$ to both sides and taking expectations. Usual stochastic order, increasing-convex order, and increasing-transform order are immediate specializations. The two displayed transform inequalities correspond to the increasing function $k\mapsto z^k$ when $z>1$ and the decreasing function $k\mapsto\e^{-sk}$ when $s>0$; the latter reverses the inequality.
\end{proof}

\begin{lemma}[Yule jump-time asymptotics]\label{lem:yule}
Let $\xi_r\sim\Exp(r)$ be independent, and put $S_\ell=\xi_1+\cdots+\xi_\ell$. Then there is a random variable $W\sim\Exp(1)$ such that
\[
        (\ell+1)\e^{-S_\ell}\to W
\]
almost surely and in $L^m$ for every fixed integer $m\ge1$. Consequently,
\[
        \frac{S_\ell}{\log\ell}\to1
\]
almost surely. Moreover, for every fixed $s>0$ and nonnegative integer $r$,
\begin{equation}\label{eq:yule-laplace}
        \E\left[\e^{-sS_\ell}S_\ell^r\right]
        =(-1)^r\frac{\dd^r}{\dd s^r}\frac{\Gamma(\ell+1)\Gamma(s+1)}{\Gamma(\ell+s+1)}.
\end{equation}
\end{lemma}

\begin{proof}
Construct the Yule pure-birth process $Z(t)$ started from $Z(0)=1$, with birth rate $n$ in state $n$. Its holding times are precisely $\xi_r\sim\Exp(r)$, and if $S_\ell$ is the time of the $\ell$th birth, then $Z(S_\ell)=\ell+1$. The process
\[
        M(t):=\e^{-t}Z(t)
\]
is a nonnegative mean-one martingale, so $M(t)\to W$ almost surely for some finite random variable $W$. Also
\[
        \Pbb(Z(t)=n)=\e^{-t}(1-\e^{-t})^{n-1},\qquad n\ge1.
\]
Therefore $\e^{-t}Z(t)\to W$ in distribution with $W\sim\Exp(1)$. The moment formula for the geometric distribution also gives convergence of all fixed moments; hence $M(t)\to W$ in $L^m$ for every fixed $m$. Evaluating at $t=S_\ell$ gives
\[
        (\ell+1)\e^{-S_\ell}=M(S_\ell)\to W
\]
almost surely. For fixed $m$,
\[
        \E\left[((\ell+1)\e^{-S_\ell})^m\right]
        = (\ell+1)^m\prod_{r=1}^\ell\frac{r}{r+m}\to m!=\E W^m,
\]
so the same convergence holds in $L^m$. Since $W>0$ almost surely,
\[
        S_\ell=\log(\ell+1)-\log((\ell+1)\e^{-S_\ell})=\log(\ell+1)+o(\log\ell)
\]
almost surely. Finally,
\[
        \E[\e^{-sS_\ell}]=\prod_{r=1}^\ell\frac{r}{r+s}
        =\frac{\Gamma(\ell+1)\Gamma(s+1)}{\Gamma(\ell+s+1)},
\]
and differentiating $r$ times with respect to $s$ gives \eqref{eq:yule-laplace}.
\end{proof}

\begin{proposition}[Joint top-spacing representation]\label{prop:top-spacing}
Fix $J\ge2$. Let $Y_1,Y_2,\ldots$ be iid rate-one Poisson processes, independent of the spacings $\xi_1,\xi_2,\ldots$, and put $S_\ell=\xi_1+\cdots+\xi_\ell$. Then, in the equal-probability model,
\[
        (U_2^N,\ldots,U_J^N)\stackrel{d}=\left(1+\sum_{\ell=1}^{N-1}\ind{Y_\ell(S_\ell)\le j-2}\right)_{j=2}^J.
\]
\end{proposition}

\begin{proof}
Use the Poissonized construction from the proof of Theorem~\ref{thm:coupling}. For the coupon whose first arrival is $\ell$ places below the maximum, the age at completion has the same law as $S_\ell$. Conditional on the first-arrival order statistics, the numbers of additional arrivals of the different ranked coupons during their ages are independent Poisson random variables with respective means equal to those ages. These counts may therefore be represented as $Y_\ell(S_\ell)$, independently over $\ell$. The coupon is missing from the $j$th album precisely when it has at most $j-2$ additional arrivals after its first arrival. The coupon whose first arrival is last contributes the leading $1$ to every $U_j^N$, $j\ge2$.
\end{proof}

\begin{lemma}[Weighted $L^m$ Toeplitz lemma]\label{lem:toeplitz}
Let $X_1,X_2,\ldots$ and $X$ be random variables with $X_\ell\to X$ in $L^m$. Let $w_{N,\ell}\ge0$, $1\le\ell<N$, satisfy
\[
        \sum_{\ell<N}w_{N,\ell}\to1,
        \qquad \sup_N\sum_{\ell<N}w_{N,\ell}<\infty,
\]
and, for every fixed $L$,
\[
        \sum_{\ell\le L}w_{N,\ell}\to0.
\]
Then
\[
        \sum_{\ell<N}w_{N,\ell}X_\ell\to X
\]
in $L^m$.
\end{lemma}

\begin{proof}
By the triangle inequality,
\[
        \left\|\sum_{\ell<N}w_{N,\ell}X_\ell-X\right\|_m
        \le\sum_{\ell<N}w_{N,\ell}\|X_\ell-X\|_m
        +\left|\sum_{\ell<N}w_{N,\ell}-1\right|\|X\|_m.
\]
Fix $L$. The contribution of $\ell\le L$ tends to zero because the total weight on this finite set tends to zero. For $\ell>L$, the contribution is at most
\[
        \left(\sup_{\ell>L}\|X_\ell-X\|_m\right)\sum_{\ell<N}w_{N,\ell}.
\]
Letting first $N\to\infty$ and then $L\to\infty$ proves the claim.
\end{proof}

\begin{lemma}[Conditional residual means]\label{lem:conditional-means}
For fixed $j\ge2$ set
\[
        a_{N,j}=\frac{(\log N)^{j-1}}{(j-1)!}
\]
and
\[
        q_{\ell,j}=\e^{-S_\ell}\sum_{m=0}^{j-2}\frac{S_\ell^m}{m!},
        \qquad Q_{N,j}=\sum_{\ell=1}^{N-1}q_{\ell,j}.
\]
Then
\[
        \frac{Q_{N,j}}{a_{N,j}}\to W
\]
almost surely and in $L^m$ for every fixed integer $m\ge1$.
\end{lemma}

\begin{proof}
The almost-sure convergence follows from Lemma~\ref{lem:yule}:
\[
        q_{\ell,j}\sim \frac{W}{\ell}\frac{(\log\ell)^{j-2}}{(j-2)!}\qquad\text{almost surely},
\]
and the elementary estimate
\[
        \sum_{\ell<N}\frac{(\log\ell)^{j-2}}{\ell}\sim\frac{(\log N)^{j-1}}{j-1}.
\]
For $L^m$ convergence, ignore finitely many initial indices and put
\[
        b_{\ell,j}=\frac{(\log\ell)^{j-2}}{(j-2)!\,\ell},
        \qquad B_{N,j}=\sum_{\ell<N}b_{\ell,j}.
\]
Then $B_{N,j}\sim a_{N,j}$. We claim that
\begin{equation}\label{eq:q-asymp-Lm}
        q_{\ell,j}=b_{\ell,j}(X_{\ell,j}+o_{L^m}(1)),
        \qquad X_{\ell,j}\to W\text{ in }L^m.
\end{equation}
For the leading term in $q_{\ell,j}$, take
\[
        X_{\ell,j}=\ell\e^{-S_\ell}\left(\frac{S_\ell}{\log\ell}\right)^{j-2},
\]
with the exponent interpreted as zero when $j=2$. Lemma~\ref{lem:yule} gives $\ell\e^{-S_\ell}\to W$ in every fixed $L^m$. Also, under the probability measure with density proportional to $\e^{-mS_\ell}$, the variables $\xi_r$ are independent exponentials with rates $r+m$; hence $S_\ell/\log\ell\to1$ in every fixed moment. Since $\ell^m\E\e^{-mS_\ell}\to\Gamma(m+1)$, equivalently by differentiating \eqref{eq:yule-laplace}, for every fixed integer $r\ge1$,
\[
        \E\left[(\ell\e^{-S_\ell})^m\left|\frac{S_\ell}{\log\ell}-1\right|^r\right]\to0.
\]
Thus $X_{\ell,j}\to W$ in $L^m$. The lower powers $S_\ell^0,\ldots,S_\ell^{j-3}$ in $q_{\ell,j}$ are smaller by powers of $\log\ell$ and contribute $o_{L^m}(b_{\ell,j})$. This proves \eqref{eq:q-asymp-Lm}.

Now write
\[
        \frac{Q_{N,j}}{a_{N,j}}=\frac{B_{N,j}}{a_{N,j}}\sum_{\ell<N}\frac{b_{\ell,j}}{B_{N,j}}X_{\ell,j}+o_{L^m}(1).
\]
The weights $b_{\ell,j}/B_{N,j}$ are nonnegative, have total mass one, and put vanishing mass on every fixed finite set of indices. Lemma~\ref{lem:toeplitz} gives
\[
        \sum_{\ell<N}\frac{b_{\ell,j}}{B_{N,j}}X_{\ell,j}\to W
\]
in $L^m$. Since $B_{N,j}/a_{N,j}\to1$, the desired $L^m$ convergence follows.
\end{proof}

\begin{theorem}[Joint fixed-index distributional and moment limit]\label{thm:joint-limit}
Fix $J\ge2$ and set $a_{N,j}=(\log N)^{j-1}/(j-1)!$. In the equal-probability model,
\[
        \left(\frac{U_2^N}{a_{N,2}},\frac{U_3^N}{a_{N,3}},\ldots,\frac{U_J^N}{a_{N,J}}\right)
        \Rightarrow (W,W,\ldots,W),\qquad W\sim\Exp(1).
\]
For every fixed nonnegative integer vector $\mathbf m=(m_2,\ldots,m_J)$,
\[
        \E\prod_{j=2}^J\left(\frac{U_j^N}{a_{N,j}}\right)^{m_j}
        \to (m_2+\cdots+m_J)!.
\]
For $J=2$ this recovers the Papanicolaou--Doumas limit theorem \cite{PapanicolaouDoumas}.
\end{theorem}

\begin{proof}
Use Proposition~\ref{prop:top-spacing}. Conditional on the spacings, the summands are independent over $\ell$, and
\[
        \Pbb(Y_\ell(S_\ell)\le j-2\mid S_1,S_2,\ldots)=q_{\ell,j}.
\]
By Lemma~\ref{lem:conditional-means}, the conditional means satisfy
\[
        \frac{Q_{N,j}}{a_{N,j}}\to W
\]
jointly, almost surely and in every fixed $L^m$.

It remains to remove the conditional Bernoulli fluctuations. Let
\[
        B_{\ell,j}=\ind{Y_\ell(S_\ell)\le j-2}.
\]
Conditional on the spacings,
\[
        \Var\left(\sum_{\ell=1}^{N-1}B_{\ell,j}\ \middle|\ S_1,S_2,\ldots\right)\le Q_{N,j}.
\]
Hence
\[
        \frac1{a_{N,j}}\sum_{\ell=1}^{N-1}(B_{\ell,j}-q_{\ell,j})\to0
\]
in $L^2$, and therefore in probability. For mixed moments, use the standard conditional Rosenthal inequality for sums of independent centered bounded variables:
\[
        \E\left[\left|\sum_{\ell=1}^{N-1}(B_{\ell,j}-q_{\ell,j})\right|^r\ \middle|\ S_1,S_2,\ldots\right]
        \le C_r(Q_{N,j}+Q_{N,j}^{r/2})
\]
for every fixed $r\ge2$. Lemma~\ref{lem:conditional-means} then gives
\[
        \frac1{a_{N,j}}\sum_{\ell=1}^{N-1}(B_{\ell,j}-q_{\ell,j})\to0\quad\text{in }L^r
\]
for every fixed $r$. Since $1/a_{N,j}\to0$, the leading last-coupon contribution is negligible on every scale $a_{N,j}$. Therefore each normalized coordinate is equal, up to an $L^r$-negligible error, to $Q_{N,j}/a_{N,j}$. The joint convergence and the mixed-moment convergence follow from Lemma~\ref{lem:conditional-means} and H\"older's inequality.
\end{proof}

\begin{corollary}[Variance, covariance, and asymptotic perfect correlation]\label{cor:correlation}
For every fixed $i,j\ge2$,
\[
        \Var(U_j^N)\sim a_{N,j}^2,
\]
\[
        \Cov(U_i^N,U_j^N)\sim a_{N,i}a_{N,j},
\]
and, for $i\ne j$,
\[
        \Corr(U_i^N,U_j^N)\to1.
\]
\end{corollary}

\begin{proof}
These are the moment cases of Theorem~\ref{thm:joint-limit} of total order at most two. Since $\E W=1$ and $\Var(W)=1$, we have
\[
        \E U_j^N\sim a_{N,j},
        \qquad \E[U_i^NU_j^N]\sim2a_{N,i}a_{N,j},
\]
which gives the covariance formula. The variance formula is the case $i=j$, and the correlation statement follows by dividing the covariance asymptotic by the product of the standard deviations.
\end{proof}

\begin{remark}[Fixed-index regime]
Theorem~\ref{thm:joint-limit} is a fixed-$J$ result. It does not address regimes in which the album index $j=j_N$ grows with $N$.
\end{remark}

\section{Concluding remarks}

We have proved three complementary facts about the siblings of the coupon collector. First, for each fixed $N$ and album index $j\ge2$, the expected number of empty spaces in collector $j$'s album is uniquely maximized by the uniform coupon distribution; in fact it decreases strictly along every nonconstant ray from the uniform vector. Second, in the equal-probability case, $U_j^N$ is stochastically increasing in $N$, with an explicit monotone coupling. Third, the same top-spacing representation gives a joint fixed-index limit theorem: after the natural normalizations, the residual counts for all fixed album indices converge to the same exponential random variable.

The extremality proof is finite-dimensional and exact, based on an alternating subset expansion and positive pair kernels. The monotonicity and limit-law proofs are probabilistic, based on the nested structure of top exponential spacings and the Yule martingale limit. For $j=2$, the distributional theorem recovers the first-sibling limit of Papanicolaou and Doumas \cite{PapanicolaouDoumas}; for any fixed $J>2$, it extends that result jointly to the first $J-1$ siblings and shows that their normalized residual counts are asymptotically perfectly correlated.

Several directions remain natural. The transfer theorem for leading expectations can likely be pushed to a second-order endpoint expansion, which should recover the full three-term expansion of Doumas and Papanicolaou from the same Poissonized calculation. It would also be interesting to determine which parts of the joint fixed-index limit persist when the album index grows with $N$, or under nonuniform probability arrays satisfying suitable endpoint hypotheses. Thus the equal-probability model is extremal at finite $N$, monotone in the alphabet size, and asymptotically governed by a single random amplitude shared by all fixed album indices.

\appendix
\section{Verification of the endpoint hypotheses for standard arrays}

\begin{proposition}[Standard endpoint arrays]\label{prop:standard-endpoint}
For every fixed $j\ge2$, Hypothesis~\ref{hyp:endpoint} holds for each of the two families
\[
        f(x)=x^\alpha,\qquad \alpha>0,
\]
and
\[
        f(x)=\e^{\lambda x^q},\qquad \lambda>0,\quad0<q<1.
\]
\end{proposition}

\begin{proof}
Put
\[
        S_{N,a}(u)=\sum_{k=1}^N r_{N,k}^a\e^{-ur_{N,k}},
        \qquad r_{N,k}=\frac{f(N)}{f(k)}.
\]
In both cases $\beta_N=f'(N)/f(N)\to0$, $\rho_N=\log(1/\beta_N)\to\infty$, and $g=f'/f$ is decreasing. Hence, for $m=N-k$,
\[
        \log r_{N,N-m}=\int_{N-m}^Ng(t)\dd t\ge \beta_Nm.
\]
We first prove the compact endpoint estimate. Uniformly for bounded $x$ and fixed $M>0$, if $0\le m\le M/(\beta_Nu_N(x))$, then
\[
        \log r_{N,N-m}=\beta_Nm+o\left(\frac1{u_N(x)}\right).
\]
Indeed, for $f(x)=x^\alpha$ this follows from
\[
        \alpha\log\frac{N}{N-m}=\frac{\alpha m}{N}+O\left(\frac{m^2}{N^2}\right),
        \qquad \beta_N=\frac\alpha N,
\]
and, for $f(x)=\e^{\lambda x^q}$, from
\[
        \lambda(N^q-(N-m)^q)=\lambda qN^{q-1}m+O(N^{q-2}m^2),
        \qquad \beta_N=\lambda qN^{q-1}.
\]
Therefore, on this window,
\[
        r_{N,N-m}^a\e^{-u_N(x)r_{N,N-m}}
        =\e^{-u_N(x)}\exp\{-u_N(x)\beta_Nm+o(1)\}.
\]
Consequently,
\[
\begin{aligned}
        \sum_{0\le m\le M/(\beta_Nu_N(x))}r_{N,N-m}^a\e^{-u_N(x)r_{N,N-m}}
        &\sim \e^{-u_N(x)}\sum_{0\le m\le M/(\beta_Nu_N(x))}\e^{-u_N(x)\beta_Nm}\\
        &\sim \frac{\e^{\rho_N-u_N(x)}}{u_N(x)}(1-\e^{-M}).
\end{aligned}
\]
The complementary part is uniformly negligible after $M\to\infty$. Since $u_N(x)\to\infty$ on compact $x$-sets and $z^a\e^{-u_N(x)z}$ is then decreasing for $z\ge1$,
\[
\begin{aligned}
        \sum_{m>M/(\beta_Nu_N(x))}r_{N,N-m}^a\e^{-u_N(x)r_{N,N-m}}
        &\le \sum_{m>M/(\beta_Nu_N(x))}\e^{a\beta_Nm}\e^{-u_N(x)\e^{\beta_Nm}}\\
        &\le C\beta_N^{-1}\e^{-u_N(x)}\int_{M/u_N(x)}^\infty \e^{ay}\e^{-u_N(x)(\e^y-1)}\dd y\\
        &\le C\e^{-cM}\frac{\e^{\rho_N-u_N(x)}}{u_N(x)}.
\end{aligned}
\]
This proves Hypothesis~\ref{hyp:endpoint}(i). The same argument with $u_N(x)$ replaced by $2u_N(x)$ and $a=0$ proves Hypothesis~\ref{hyp:endpoint}(ii).

For the right-tail envelope, take $u=u_N(x)$ with $x\ge0$. The previous monotone comparison and a Riemann-sum bound give, for fixed $j\ge2$,
\[
        S_{N,j}(u)\le \sum_{m=0}^{N-1}\e^{j\beta_Nm}\e^{-u\e^{\beta_Nm}}
        \le C\e^{-u}+C\beta_N^{-1}\int_0^\infty \e^{jy}\e^{-u\e^y}\dd y
        \le C_j\left(\e^{-u}+\frac{\e^{\rho_N-u}}{u}\right).
\]
Here the last integral is at most $C_j\e^{-u}/u$. Thus Hypothesis~\ref{hyp:endpoint}(iii) holds.

It remains to verify the left-tail envelope. Write $u=u_N(x)$, so $\dd x=\dd u$, and split the region $x<-L$, $u>0$, into $u\ge1$ and $0<u<1$.

First let $u\ge1$. In the range $1\le u\le\rho_N-\log\rho_N-L$, the additive $\e^{-u}$ term in the preceding bound is absorbed by $\beta_N^{-1}\e^{-u}/u$. Also, for a small fixed $c_0>0$ and all $0\le m\le c_0/(\beta_Nu)$, the explicit formulae above give $r_{N,N-m}\le1+C/u$. Hence
\[
        S_{N,0}(u)-\e^{-u}\ge c\frac{\e^{-u}}{\beta_Nu}.
\]
The integrand in Hypothesis~\ref{hyp:endpoint}(iv) is therefore bounded by
\[
        C\left(\frac{u}{\rho_N}\right)^{j-1}\frac{\e^{\rho_N-u}}{u}
        \exp\left\{-c\frac{\e^{\rho_N-u}}{u}\right\}.
\]
Since $\e^{\rho_N-u}/u=(\rho_N/u)\e^{-x}$ and $u\le\rho_N-\log\rho_N-L$ on the left tail, this is dominated by
\[
        C\e^{-x}\exp\{-c\e^{-x}\},
\]
whose integral over $(-\infty,-L)$ tends to zero as $L\to\infty$.

Now let $0<u<1$. Put $R_N=r_{N,1}$ and
\[
        A_N(y)=\#\{k:r_{N,k}\le y\},\qquad 1\le y\le R_N.
\]
For the power array,
\[
        A_N(y)=N-\lceil Ny^{-1/\alpha}\rceil+1,
\]
whereas for the stretched-exponential array,
\[
        A_N(y)=N-\left\lceil\left(N^q-\frac{\log y}{\lambda}\right)_+^{1/q}\right\rceil+1.
\]
These formulae imply the saturated counting estimates
\[
        c\min\{N,\beta_N^{-1}\log y\}-C
        \le A_N(y)\le C\left(1+\min\{N,\beta_N^{-1}\log y\}\right).
\]
For $1/R_N\le u\le1$, the indices with $r_{N,k}\le1/u$ contribute at least $\e^{-1}$ to $S_{N,0}(u)$, and hence
\[
        S_{N,0}(u)-\e^{-u}\ge c\left(\min\{N,\beta_N^{-1}\log(1/u)\}-C\right)_+.
\]
The same explicit inverses, viewed as Stieltjes changes of variables, give
\[
        S_{N,j}(u)\le C_j\beta_N^{-1}u^{-j},\qquad 1/R_N\le u\le1.
\]
Indeed, in the power case the sum is bounded by a constant times $N\int_1^{R_N}y^{j-1/\alpha-1}\e^{-uy}\dd y$, and in the stretched-exponential case the inverse change of variables has derivative bounded by $C\beta_N^{-1}/y$. Thus the contribution from $1/R_N\le u\le1$ is at most
\[
        C_j\rho_N^{-(j-1)}\beta_N^{-1}\int_0^{\log R_N}\exp\{-c(\min\{N,\beta_N^{-1}v\}-C)_+\}\dd v,
\]
after the substitution $v=\log(1/u)$. The integral is $O(\beta_N)+O((\log R_N)\e^{-cN})$ for the two arrays, and so this part is $O(\rho_N^{-(j-1)})=o(1)$.

Finally, if $0<u<1/R_N$, then $ur_{N,k}\le1$ for every $k$, so $S_{N,0}(u)\ge \e^{-1}N$. Also
\[
        \sum_{k=1}^Nr_{N,k}^j\le C_j\beta_N^{-1}R_N^j
\]
for both arrays, by the same inverse comparisons. Hence this remaining part is bounded by
\[
        C_j\rho_N^{-(j-1)}\beta_N^{-1}R_N^j\e^{-cN}\int_0^{1/R_N}u^{j-1}\dd u
        \le C_j\rho_N^{-(j-1)}\beta_N^{-1}\e^{-cN},
\]
which tends to zero. Combining the three regions proves Hypothesis~\ref{hyp:endpoint}(iv) and completes the proof.
\end{proof}

\end{document}